\documentclass[12pt]{amsart}
\usepackage[margin=1in]{geometry}

\usepackage{amssymb,amsfonts}
\usepackage[utf8]{inputenc}
\usepackage{amsmath}
\usepackage[all,arc]{xy}
\usepackage{relsize}
\usepackage{mathrsfs}
\usepackage[mathscr]{euscript}
\usepackage{relsize}
\usepackage{mathtools}
\usepackage{mathabx,epsfig}
\usepackage[normalem]{ulem}
\usepackage{todonotes}
\usepackage{hyperref}

\usepackage{graphicx}
\graphicspath{ {images/} }
\usepackage{tikz}
\usepackage{changepage}
\usetikzlibrary{arrows,positioning, calc}
\tikzstyle{vertex}=[draw,fill=black!8,circle,minimum size=6pt,inner sep=0pt]
\usepackage{enumerate,cite}

\newtheorem{thm}{Theorem}[section]
\newtheorem{cor}[thm]{Corollary}

\theoremstyle{definition}

\newtheorem{con}[thm]{Construction}
\newtheorem{exmp}[thm]{Example}

\newtheorem{rmk}[thm]{Remark}
\usepackage{bm}

\def\m{\bm{m}}

\def\bb{\bm{b}}
\def\x{\bm{x}}
\def\y{\bm{y}}
\def\1{\mathbf{1}}
\def\A{\mathbf{A}}
\def\0{\mathbf{0}}
\def\F{\mathcal{F}}
\def\p{\bm{p}}
\def\v1{\bm{v}_1}
\def\vi{\bm{v}_i}
\def\vj{\bm{v}_j}
\def\vvd{\bm{v}_d}
\def\vd1{\bm{v}_{d+1}}
\def\vvr{\bm{v}_r}
\def\vr1{\bm{v}_{r+1}}

\def\ed{\bm{e_d}}

\def\i{\textit}

\def\Q{\mathbb{Q}}
\def\R{\mathbb{R}}
\def\Z{\mathbb{Z}}

\newcommand\hstar{h^\ast}
\usepackage{tikz}

\newcommand{\conv}{\mathrm{conv}}
\newcommand{\vol}{\mathrm{vol}}
\newcommand{\Ehr}{\operatorname{Ehr}}
\newcommand{\ehr}{\operatorname{ehr}}
\newcommand{\Pyr}{\operatorname{Pyr}}
\newcommand{\homg}{\operatorname{hom}}
\newcommand{\rationalehr}{\operatorname{ehr}_\mathbb{Q}} 
\newcommand{\rationalEhr}[2]{\operatorname{Ehr}_\mathbb{Q}\left(#1;#2\right)} 
\newcommand{\refinedrationalEhr}[2]{\operatorname{Ehr}_\mathbb{Q}^{\sf{ref}}\left(#1;#2\right)}
\newcommand{\rationalhstar}[3]{%
      {h^*_{\mathbb{Q}}(#1;#2;#3)}
}
\newcommand{\refinedrationalhstar}[3]{%
      {h^{*\sf{ref}}_{\mathbb{Q}}(#1;#2;#3)}
}

\usepackage{paralist}

\makeatletter
\let\c@equation\c@thm
\makeatother
\numberwithin{equation}{section}

\title{Boundary $\hstar$-polynomials of rational polytopes}

\author{Esme Bajo}
\address{Department of Mathematics, UC Berkeley}
\email{esme@berkeley.edu}

\author{Matthias Beck}
\address{Department of Mathematics, San Francisco State University}
\email{becksfsu@gmail.com}

\date{20 January 2023}
\thanks{We are indebted to three anonymous referees for numerous helpful comments.}

\begin{document}

\maketitle

\begin{abstract}
If $P$ is a \emph{lattice polytope} (i.e., $P$ is the convex hull of finitely many integer points in $\R^d$) of dimension $d$, Ehrhart's famous theorem (1962) asserts that the integer-point counting function $|nP \cap \Z^d|$ is a degree-$d$ polynomial in the integer variable $n$. Equivalently, the generating function $1 + \sum_{n \ge 1} |nP \cap \Z^d| \, z^n$ is a rational function of the form $\frac{ \hstar(z) }{ (1-z)^{ d+1 } }$; we call $\hstar(z)$ the \emph{$\hstar$-polynomial} of $P$. There are several known necessary conditions for $\hstar$-polynomials, including results by Hibi (1990), Stanley (1991), and Stapledon (2009), who used an interplay of arithmetic (integer-point structure) and topological (local $h$-vectors of triangulations) data of a given polytope. We introduce an alternative \emph{ansatz} to understand Ehrhart theory through the $\hstar$-polynomial of the \emph{boundary} of a polytope, recovering all of the above results and their extensions for rational polytopes in a unifying manner.  We include applications for (rational) Gorenstein polytopes and rational Ehrhart dilations.
\end{abstract}

\setlength{\abovedisplayskip}{3pt}
\setlength{\belowdisplayskip}{3pt}


\section{Introduction}\label{sec:intro}

For a $d$-dimensional \emph{rational polytope} $P \subset \R^{d}$ (i.e., the convex hull of finitely
many points in $\Q^d$) and a positive integer $n$, let $\ehr_P(n)$ denote the number of lattice points in $nP$. 
Ehrhart's famous theorem~\cite{ehrhartpolynomial} says that
\begin{equation}\label{eq:quasidef}
  \ehr_P(n) = \vol(P) \, n^d+k_{d-1}(n) \, n^{d-1}+\cdots+k_1(n)\, n+k_0(n)\, , 
\end{equation}
where $k_0(n),k_1(n),\dots,k_{d-1}(n)$ are periodic functions in $n$.
We call $\ehr_P(n)$ the \emph{Ehrhart quasipolynomial} of $P$. Ehrhart proved that each
period of $k_0(n),k_1(n),\dots,k_{d-1}(n)$ divides the \emph{denominator} of $P$, which is the smallest
integer $q$ such that $q P$ is a \emph{lattice polytope} (i.e., the convex hull of
points in $\Z^d$). 
The \emph{Ehrhart series} is the rational generating function
\[
\Ehr_P(z) := 1 + \sum_{n\geq1} \ehr_P(n)z^n=\frac{\hstar_P(z)}{(1-z^q)^{d+1}} \, ,
\]
where $\hstar_P(z)$ is a polynomial of degree less than $q(d+1)$, the
\emph{$\hstar$-polynomial} of~$P$ (see, e.g., \cite[Section 3.8]{ccd}). 
The best known instance is when $q = 1$, i.e., $P$ is a lattice polytope, in which case $\ehr_P(n)$ is a polynomial. 

Note that the $\hstar$-polynomial depends not only on $q$ (though that is implicitly determined by $P$), but
also on our choice of representing the rational function $\Ehr_P(z)$, which in our form
will not be in lowest terms. For example, the Ehrhart series of the line segment
$P=[- \frac 1 2, \frac 1 2 ]$ is 
\[
\Ehr_P(z)=\frac{1+z+z^2+z^3}{(1-z^2)^2}=\frac{1+z^2}{(1-z)(1-z^2)},
\]
but we typically insist that $h^\ast_P(z)=1+z+z^2+z^3$. Occasionally (e.g., in
Corollary~\ref{cor:rationalpyramids} below) we use different
representations.

Stanley proved that $\hstar_P(z)$ has nonnegative integer coefficients. While
his original proof used commutative algebra~\cite{stanleymagiccohenmac} based on work
of Hochster~\cite{hochster}, and a second proof of Stanley proceeded via shelling
arguments~\cite{stanleydecomp}, there is an ``easy'' conceptual proof using half-open triangulations (see, e.g., \cite[Chapter 5]{crt}).
Stanley's theorem was refined by Stapledon~\cite{stapledondelta} (for lattice polytopes, i.e., $q=1$) and 
Beck--Braun--Vindas-Mel\'{e}ndez~\cite{beckrational} (for rational polytopes), who showed
that the $h^\ast$-polynomial can be decomposed using palindromic polynomials with
nonnegative coefficients. A \i{palindromic polynomial} $f(z)=\sum f_iz^i$ has symmetric
coefficients, that is, $f_i=f_{\text{deg}(f)-i}$ for $i=0,\hdots,\text{deg}(f)$;
equivalently, $z^{\text{deg}(f)}f(\frac 1 z)=f(z)$.
\begin{thm}\label{thm:stapledonrat}
Let $P$ be a full-dimensional rational polytope with denominator $q$ and let $\ell$ be the smallest positive integer such that $\ell P$ contains an interior lattice point. Then
\[
\frac{1+z+\dots+z^{\ell-1}}{1+z+\dots+z^{q-1}}\, h_P^\ast(z)=a(z)+z^\ell \, b(z) \, ,
\]
where $a(z)$ and $b(z)$ are \i{palindromic} polynomials with nonnegative integer coefficients.
\end{thm}

We remark that our assumption that $P$ is full dimensional and our
convention for $\hstar_P(z)$ imply that $\hstar_P(z)$ is divisible by $1+z+\dots+z^{q-1}$, because the
leading coefficient of $\ehr_P(n)$ is constant (namely, the volume of $P$) and thus $z=1$
is the unique pole with maximal order of the rational function $\Ehr_P(z)$; see, e.g.,
\cite[Theorem~4.1.1]{stanleyec1}.
Thus the decomposition of the polynomial
$\frac{1+z+\dots+z^{\ell-1}}{1+z+\dots+z^{q-1}}\, h_P^\ast(z)$ into two palindromic
polynomials as in Theorem~\ref{thm:stapledonrat} is unique and an easy exercise. The point of
Theorem~\ref{thm:stapledonrat} is that $a(z)$ and $b(z)$ have nonnegative coefficients.

As was pointed out in~\cite{beckrational},
Theorem~\ref{thm:stapledonrat} immediately implies the inequalities
\begin{align*}
      \hstar_0+\cdots+\hstar_{j+1}\geq \hstar_{q(d+1)-1}+\cdots+\hstar_{q(d+1)-1-j} \, , \qquad &j=0,\dots,
\left\lfloor \tfrac{q(d+1)-1}{2} \right\rfloor -1 \, , \\
    \hstar_s+\cdots+\hstar_{s-j}\geq \hstar_0+\cdots+\hstar_j \, , \qquad &j=0,\dots,s \, ,
\end{align*}
where $s:=\deg{\hstar_P(z)}$. For the case $q = 1$ (i.e., lattice polytopes), they go back to Hibi~\cite{hibiehrhartineq,hibilowerbound} and Stanley~\cite{stanleyinequ}.

The case $q = \ell = 1$ of Theorem~\ref{thm:stapledonrat} was proved much earlier by Betke--McMullen~\cite{betkemcmullen}.
Writing a (combinatorially defined) polynomial as a sum of two palindromic polynomials as in
Theorem~\ref{thm:stapledonrat} is often referred to as a \textit{symmetric decomposition}
and has applications beyond a refinement of nonnegativity; we mention one representative
(much more can be found, e.g., in \cite{brandensolus}): $p(z) = a(z) + z \, b(z)$, with
$a(z)$ and $b(z)$ palindromic, is alternatingly increasing (i.e., the coefficients of $p(z)$
satisfy $0 \le p_0 \le p_d \le p_1 \le p_{ d-1 } \le \cdots$) if and only if the
coefficients of both $a(z)$ and
$b(z)$ are nonnegative and unimodal (i.e., the coefficients increase up to some index and
then decrease).

Betke--McMullen's, Stapledon's, as well as Beck--Braun--Vindas-Mel\'{e}ndez's proofs use local
$h$-vectors of a triangulation, their nonnegativity, and the Dehn--Sommerville relations.
Our main goal is to give an ``easy'' conceptual proof of Theorem~\ref{thm:stapledonrat}---in
particular, one that is independent of local $h$-vectors (though they are
hiding under the surface).
Our \emph{ansatz} is to study the $\hstar$-polynomial of the \emph{boundary} of a rational polytope,
and our second goal is to exhibit that such a study is worthwhile, with the hope for further
applications.
The connection to Theorem~\ref{thm:stapledonrat} is that the $\hstar$-polynomial of the boundary of
$P$, defined via
\begin{equation}\label{eq:boundaryhstardef}
\Ehr_{\partial P}(z) := 1 + \sum_{n\geq1} \ehr_{\partial P}(n)z^n=\frac{\hstar_{\partial P}(z)}{(1-z^q)^{d}} \, ,
\end{equation}
equals $a(z)$, for any $q$ and $\ell$;\footnote{
We suspect that this fact is well known to the experts, but we could not find it in the literature. We
thank Katharina Jochemko for pointing it out to us.
}
as we will see below, this equality follows from the uniqueness of the
symmetric decomposition and the palindromicity~\eqref{eq:hstarpalindromic} of $h^\ast_{\partial P}(z)$.
We remark that, a priori, it is not clear that we can always represent the generating
function of $\ehr_{\partial P}(n)$ in the form~\eqref{eq:boundaryhstardef}; in
particular, we will see below (where we will show that this form always exists) that 
$\hstar_{\partial P}(z)$ has degree $qd$, contrary to the degree of
$\hstar_{P}(z)$, and so the quasipolynomial $\ehr_{\partial P}(n)$ does not have constant
term~1.

To illustrate the philosophy behind our approach, here is a do-it-yourself proof setup for Theorem~\ref{thm:stapledonrat} in the case $q = \ell = 1$:
\begin{itemize}
\item fix a (half-open) triangulation $T$ of the boundary $\partial P$ and extend $T$ to a (half-open) triangulation of $P$ by coning over an interior lattice point $\x$;
\item convince yourself that $a(z) = \hstar_{\partial P}(z)$ is palindromic with
nonnegative (in fact, as we will show in Theorem~\ref{thm:lowerbound}, positive) coefficients;
\item realize that the $\hstar$-polynomial of each half-open simplex $\Delta \in T$ is
coefficient-wise less than or equal to the $\hstar$-polynomial of $\conv(\Delta, \x)$, and thus $\hstar_P(z) - a(z)$
has nonnegative coefficients.
\end{itemize}
It turns out that this philosophy works for general $q$ and $\ell$, with slight modifications
(for example, for $\ell > 1$, the interior point $\x$ cannot be a lattice point).
To state our \emph{ansatz} from a different angle, we approach
Theorem~\ref{thm:stapledonrat} by giving (1) a (positive, symmetric) interpretation
of $a(z)$ and (2) an explicit construction which shows that $a(z) \le 
\frac{1+z+\dots+z^{\ell-1}}{1+z+\dots+z^{q-1}}\, h_P^\ast(z)$.

This point of view has consequences beyond a (somewhat short) proof of Theorem~\ref{thm:stapledonrat}.
One of these consequences is an inequality among the two leading coefficients of an Ehrhart polynomial
which seems to be novel.

\begin{cor}\label{cor:twoleadingcoeff}
Let $P$ be a full-dimensional lattice $d$-polytope and let $\ell$ be the smallest positive integer such
that $\ell P$ contains an interior lattice point. Then the two leading coefficients of 
$\ehr_P(n) = k_d \, n^d+k_{d-1} \, n^{d-1}+\cdots+k_0$ satisfy
\[
  \frac{ \ell \, d }{ 2 } \, k_d \ge k_{d-1} \, .
\]
\end{cor}

Betke--McMullen~\cite[Theorem~6]{betkemcmullen} gave upper bounds for each $k_j$ in terms of $k_d$ and
Stirling numbers of the first kind. For $j = d-1$ they yield $\binom{d+1} 2 \, k_d \ge k_{ d-1 }$, which
Corollary~\ref{cor:twoleadingcoeff} improves upon.\footnote{
We thank Martin Henk for reminding us about the Betke--McMullen inequalities and asking
whether they could be improved using our setup.
}

The structure of the paper is as follows.
Section~\ref{sec:boundary} serves as a point of departure for our study of $\hstar_{\partial P}(z)$.
We prove several inequalities for the coefficients of $\hstar_{\partial P}(z)$, among them a
lower bound result (Theorem~\ref{thm:lowerbound}), which in particular shows that $\hstar_{\partial P}(z)$ has \emph{positive} coefficients.

Section~\ref{sec:triangulations} gives a construction for certain half-open triangulations we will
need for our proof of Theorem~\ref{thm:stapledonrat} in Section~\ref{sec:mainproof}; this section also
contains a proof of Corollary~\ref{cor:twoleadingcoeff}.

In Section~\ref{sec:gorenstein} we discuss reflexive and Gorenstein polytopes, as well as their
rational analogues, and implications for these polytopes from the viewpoint of $\hstar_{\partial
P}(z)$. Our main result in this section (Theorem~\ref{thm:rationalgorenteinboundary}) says that if 
$P$ is a rational polytope with denominator $q$ such that $gP$ is reflexive, then
\[
\hstar_P(z)=\frac{1+z+\dots+z^{q-1}}{1+z+\dots+z^{g-1}} \, \hstar_{\partial P}(z) \, .
\]

In Section~\ref{sec:ratehrhart} we prove an analogue of Theorem~\ref{thm:stapledonrat} for rational
(or, equivalently, real) Ehrhart dilations.


\section{The $\hstar$-polynomial of the Boundary of a Polytope}\label{sec:boundary}

We start by addressing some of the subtleties in defining the Ehrhart series (and thus the $\hstar$-polynomial) of $\partial P$, especially when we ultimately compute the Ehrhart series using a half-open triangulation of the boundary. 
The reason for the convention \eqref{eq:boundaryhstardef}
that $\Ehr_{\partial P}(z)$ (and therefore also $\hstar_{\partial
P}(z)$) has constant term 1 is \emph{Ehrhart--Macdonald reciprocity} (see, e.g.,~\cite[Corollary~5.4.5]{crt}): it says that the rational generating functions $\Ehr_{P}(z)$ and
\[
\Ehr_{P^\circ}(z) := \sum_{n\geq1} \ehr_{P^\circ}(n)z^n=\frac{\hstar_{P^\circ}(z)}{(1-z^q)^{d+1}}
\]
are related via
\[
\Ehr_{P^\circ}(\tfrac 1 z)=(-1)^{d+1}\Ehr_P(z)
\]
or, equivalently,
\[
  z^{q(d+1)} \, \hstar_{P^\circ}(\tfrac 1 z)= \hstar_P(z) \, .
\]
As
$
\Ehr_{\partial P}(z)= \Ehr_P(z)-\Ehr_{P^\circ}(z),
$
we obtain
\[
\hstar_{\partial P}(z)=\frac{\hstar_P(z)-\hstar_{P^\circ}(z)}{1-z^q}
\]
and thus $h^\ast_{\partial P}(z)$ is palindromic, i.e.,
\begin{equation}\label{eq:hstarpalindromic}
  z^{qd} \, \hstar_{\partial P}(\tfrac 1 z) = \hstar_{\partial P}(z) \, .
\end{equation}
(See~\cite[Proposition~5.6.3]{crt} for connections to more general self-reciprocal complexes and the
Dehn--Sommerville relations.)

So $\hstar$-polynomials of boundaries of polytopes are in a sense more restricted than
$\hstar$-polynomials of polytopes. Even further, as we will show in Theorem~\ref{thm:lowerbound} below,
$\hstar_{\partial P}(z)$ has no internal zero coefficients, which is far
from true for $\hstar_{P}(z)$ (see, e.g., \cite{higashitanicounterex}).

On the other hand, there is the following alternative extension of nonnegativity of $\hstar_{P}(z)$ due to Stanley~\cite{stanleymonotonicity} (we state only the version for lattice polytopes):

\begin{thm}
Let $P$ and $Q$ be lattice polytopes with $P\subseteq Q$. Then $h_{P}^\ast(z) \leq h_{Q}^\ast(z)$
coefficient-wise.
\end{thm}

Monotonicity does not hold for $\hstar_{\partial P}(z)$, as the following example exhibits.

\begin{exmp}
Let $P=\conv\{(0,0),(0,2),(2,0),(2,2)\}$
and $Q=\conv\{(0,0),(0,2),(2,0),(3,3)\}$.
Then
\[
\hstar_{\partial P}(z) 
=1+6z+z^2
\qquad \text{ and } \qquad
\hstar_{\partial Q}(z) 
=1+4z+z^2.
\]
Thus $P\subseteq Q$ but $\hstar_{\partial P}(z)\geq \hstar_{\partial Q}(z)$ coefficient-wise.
\end{exmp}

Moreover, even though $\partial P\subseteq P$, it is not always true that $\hstar_{\partial P}(z)\leq
\hstar_P(z)$: while $\hstar_{\partial P}(z)$ always has degree $qd$, there are many polytopes for
which $\hstar_P(z)$ has lower degree. On the other hand, the $\ell=q=1$ case of Theorem~\ref{thm:stapledonrat} implies 
\[
\hstar_P(z)=\hstar_{\partial P}(z)+z\, b(z) \, ,
\] 
thus when $P$ is a lattice polytope containing an interior lattice point, it is true that
$\hstar_{\partial P}(z)\leq \hstar_P(z)$. The rational version gives us a more general
necessary condition in terms of $\ell$ and $q$ for $\hstar_{\partial P}(z)\leq
\hstar_P(z)$ to hold (again, assuming Theorem~\ref{thm:stapledonrat}):

\begin{cor}
Let $P\subset \R^d$ be a rational polytope with denominator $q$, and let $\ell\geq1$ be
the smallest dilate of $P$ that contains an interior lattice point. If $\ell\leq q$, then
$\hstar_{\partial P}(z)\leq \hstar_P(z)$. (In particular, when $P$ is a lattice polytope, $\ell=1$ suffices.)
\end{cor}
\begin{proof}
By Theorem~\ref{thm:stapledonrat} (with the interpretation $a(z) = \hstar_{\partial P}(z)$) 
and the assumption $\ell\leq q$,
\[
\hstar_{\partial P}(z) \le
(1+z+\dots+z^{\ell-1})\, \frac{\hstar_P(z)}{1+z+\dots+z^{q-1}} \le
(1+z+\dots+z^{q-1})\, \frac{\hstar_P(z)}{1+z+\dots+z^{q-1}}
\]
coefficient-wise. 
\end{proof}

The following lower-bound result is a 
restatement of a theorem of Stapledon~\cite[Theorem~2.14]{stapledondelta} (and a simplified proof) in our language:

\begin{thm}\label{thm:lowerbound}
If $P$ is a lattice $d$-polytope with boundary $\hstar$-polynomial $\hstar_{\partial P}(z)=\hstar_{\partial P,d}z^d+\dots+\hstar_{\partial P,1}z+\hstar_{\partial P,0}$ then
\[
  1 = \hstar_{\partial P,0} \le \hstar_{\partial P,1}\le \hstar_{\partial P,j}
  \qquad \text{ for } j=2,\dots,d-1. 
\]
In particular, $\hstar_{\partial P}(z)$ has positive coefficients.
\end{thm}

This result parallels a lower-bound theorem of
Hibi~\cite{hibilowerbound}, who proved that if $P$ is a $d$-dimensional lattice polytope with
$\hstar$-polynomial $h_P^\ast(z)=\hstar_{P,d}z^d+\dots+\hstar_{P,1}z+\hstar_{P,0}$ of degree $d$
(i.e., $P$ contains an interior lattice point), then $1 = h_{P,0}^\ast\leq h_{P,1}^\ast\leq \hstar_{P,j}$ for $j=2,\dots,d-1$.
Our proof mirrors that of Hibi; it turns out that adapting it for $h_{\partial P}^\ast(z)$ simplifies the proof.

\begin{proof}[Proof of Theorem~\ref{thm:lowerbound}]
Let $T$ be a triangulation of $\partial P$ that uses every lattice point in $\partial P$
(i.e., every lattice point in $\partial P$ is a vertex of a simplex in~$T$),
with $h$-vector $(h_0,\dots,h_d)$ defined, as usual, via
\[
  h_d \, x^d + h_{ d-1 } \, x^{ d-1 } + \dots + h_0 = \sum_{ k = -1 }^{ d-1 } f_k \,
x^{ k+1 } (1-x)^{ d-1-k } 
\]
where $f_k$ denotes the number of $k$-simplices in $T$ and $f_{ -1 } = 1$. 
The definitions of $h$ and $\hstar$ imply
\begin{equation}\label{eq:ineq1}
  h_{1} = f_0-d = \left| \partial P \cap \Z^d \right| - d = \hstar_{\partial P,1} \, .
\end{equation}
(Stanley~\cite{stanleycombcommalg} proved much more about relations between $h$ and
$\hstar$.)
Barnette's famous lower bound theorem~\cite{barnette} (see also \cite[Theorem 13.1]{hibi}) says
\begin{equation}\label{eq:ineq2}
  h_1 \le h_j
  \qquad \text{ for } j = 2, \dots, d-1.
\end{equation}
Finally, we apply a theorem of Betke and McMullen~\cite[Theorem~2]{betkemcmullen} to our
situation: it yields
\begin{equation}\label{eq:ineq3}
   h_{ j} \le \hstar_{\partial P, j}
  \qquad \text{ for } j = 0, \dots, d.
\end{equation}
(Betke--McMullen \cite{betkemcmullen} prove much more, giving a formula for $\hstar$ in terms
of local $h$-vectors.)
The inequalities \eqref{eq:ineq1}, \eqref{eq:ineq2}, and \eqref{eq:ineq3} line up to
complete our proof.
\end{proof}

We finish this section by recalling, for the record, another set of inequalities for
$\hstar_{ \partial P } (z)$ in the special (and important) case that $P$ admits a regular
unimodular boundary triangulation, once more due to Stapledon~\cite[Theorem~2.20]{stapledondelta}:
\[
  1 = \hstar_{\partial P, 0} \le \hstar_{\partial P, 1} \le \dots \le \hstar_{\partial P,
\lfloor \frac d 2 \rfloor } 
  \qquad \text{ and } \qquad
  \hstar_{\partial P, j} \le \binom{ \hstar_{\partial P, 1} + j - 1 }{ j } \, .
\]


\section{Half-open Triangulations}\label{sec:triangulations}

Our approach is to triangulate $\partial P$ into disjoint half-open simplices of dimension $d-1$, in order to
avoid inclusion--exclusion arguments. There is a subtlety stemming from our convention that the $\hstar$-polynomial of $\partial P$ has constant term 1, which we need to address here.

Before introducing the half-open boundary triangulations we will use in the later proofs, we recall the Ehrhart series of a \textit{half-open} simplex. Let $\Delta\subseteq \R^d$ be the simplex with vertices $\v1,\hdots,\vd1\in\frac{1}{q}\Z^d$, where the facets opposite $\v1,\hdots,\vvr$ are missing. That is, let
\[
 \Delta = \left\{ \lambda_1 \v1 + \dots + \lambda_{d+1} \vd1 : 
  \begin{array}{l}
  \lambda_1, \dots, \lambda_{r} > 0 \\
  \lambda_{r+1},\dots,\lambda_{d+1} \ge 0 \\
  \lambda_1 + \dots + \lambda_{d+1} = 1
  \end{array}
  \right\} \, .
\]
Its Ehrhart series is
\[
\Ehr_\Delta(z):=\sum_{n\geq0} \ehr_\Delta(n)z^n,
\] 
with constant term 1 if and only if $\Delta$ is closed, therefore it is possible for its expression as a rational function to be improper. That is, as seen in the construction below, the Ehrhart series can still be expressed in the form
\[
\Ehr_\Delta(z)=\frac{\hstar_\Delta(z)}{(1-z^q)^{d+1}},
\]
but it is possible for the $\hstar$-polynomial to have degree equal to $q(d+1)$. The
following tiling argument of the homogenization of a rational simplex will also be crucial to the proof of Theorem~\ref{thm:stapledonrat} in the next section.

\begin{con}\label{con:tiling}
We follow~\cite[Section~4.6]{crt} and define the \i{homogenization} of $\Delta$ to be the half-open cone
\[
 \homg(\Delta) := \sum_{ j=1 }^{ r } \R_{ >0 } \begin{pmatrix}\vj\\1\end{pmatrix} + \sum_{ j=r+1 }^{ d+1 } \R_{ \ge 0 } \begin{pmatrix}\vj\\1\end{pmatrix} \subset \R^{ d+1 } .
\]
Observe that the intersection of $\homg(\Delta)$ with the hyperplane $x_{d+1}=n$ gives
precisely a copy of the $n$th dilate of $\Delta$, whence
\[
\Ehr_{\Delta}(z)=\sum_{\x\in \homg(\Delta)\cap \Z^{d+1}}z^{x_{d+1}}.
\]
Define the \i{fundamental parallelepiped} of $\Delta$ to be
\[
 \F_\Delta := \sum_{ j=1 }^{ r } (0,1] \begin{pmatrix}q\vj\\q\end{pmatrix} + \sum_{ j=r+1 }^{
d+1 } [0,1) \begin{pmatrix}q\vj\\q\end{pmatrix}.
\]
Because the generators of $\homg(\Delta)$ are linearly independent, $\homg(\Delta)$ can
be tiled with translates of $\F_\Delta$. More precisely, every
point in $\homg(\Delta)$ can be uniquely expressed as the sum of a nonnegative integral
combination of the $\binom{q\vj}{q}$ and an integral point in $\F_\Delta $. This yields
\[
\Ehr_\Delta(z)=\sum_{k_1,\dots,k_{d+1}\in\Z_{\geq0}}\sum_{\m\in\F_\Delta\cap\Z^{d+1}}z^{qk_1+\dots+qk_{d+1}+m_{d+1}}=\frac{\sum_{\m\in
\F_\Delta\cap\Z^{d+1}}z^{m_{d+1}}}{(1-z^q)^{d+1}} \, .
\]
Observe that we chose the fundamental parallelepiped to have generators each with last coordinate $q$, however it is not necessary for this to be the case. For example, suppose $q_1,\dots,q_{d+1}$ are such that $\vj\in \frac{1}{q_j}\Z^{d+1}$. Then we can tile $\homg(\Delta)$ with translates of the fundamental parallelepiped 
\[
\F'_\Delta := \sum_{ j=1 }^{ r } (0,1] \begin{pmatrix}q_j\vj\\q_j\end{pmatrix} + \sum_{ j=r+1 }^{
d+1 } [0,1) \begin{pmatrix}q_j\vj\\q_j\end{pmatrix}
\]
to obtain the expression for the Ehrhart series
\[
\Ehr_\Delta(z)=\sum_{k_1,\dots,k_{d+1}\in\Z_{\geq0}}\sum_{\m\in\F'_\Delta\cap\Z^{d+1}}z^{q_1k_1+\dots+q_{d+1}k_{d+1}+m_{d+1}}=\frac{\sum_{\m\in
\F'_\Delta\cap\Z^{d+1}}z^{m_{d+1}}}{(1-z^{q_1})\cdots(1-z^{q_{d+1}})} \, .
\]
Therefore, the denominator in our resulting expression of the Ehrhart series depends on our choice of fundamental parallelepiped, a subtlety that will be important in the following section.
\end{con}

Naturally, we would like to construct a triangulation of $\partial P$ such that
\begin{equation}\label{eq:halfopengood}
\hstar_{\partial P}(z)=\sum_{\Delta\in T}\hstar_\Delta(z)
\end{equation}
when $T$ is a disjoint half-open triangulation of~$\partial P$. We achieve this as follows:
\begin{itemize}
\item For a $d$-dimensional rational polytope $P$ with denominator $q$, choose a rational triangulation $T$ of $\partial P$ into $(d-1)$-dimensional
simplices with denominators dividing $q$. This can be accomplished, for example, by picking a triangulation that uses only the vertices of $P$.
\item Use the convention to express the Ehrhart series as if each simplex $\Delta\in T$
has denominator $q$, i.e., define
\[
\hstar_\Delta(z)=(1-z^q)^d \Ehr_\Delta(z) \, .
\]
This guarantees that we will already have a common denominator when adding up the Ehrhart series of the simplices in the boundary triangulation.
\item Apply Construction~\ref{con:goodtriang} below to turn $T$ into a disjoint triangulation with exactly one closed simplex. This guarantees that the constant term of $\Ehr_{\partial P}(z)$ will match the constant term of the sum of the Ehrhart series of the simplices in $T$.
\end{itemize}

Once we are able to construct such a disjoint boundary triangulation $T$ with exactly one closed simplex, we will obtain
\begin{align*}
\sum_{\Delta \in T} \frac{\hstar_\Delta(z)}{(1-z^q)^d}=\sum_{n\geq0}\sum_{\Delta\in
T}\ehr_{\Delta}(n)z^n=1+\sum_{n\geq1}\ehr_{\partial P}(n)=\frac{h^\ast_{\partial
P}(z)}{(1-z^q)^d}\, ,
\end{align*}
as desired.

Fortunately, the following construction gives a disjoint boundary triangulation $T$ into half-open
$(d-1)$-simplices, exactly one of which is closed; for such a $T$, we will have~\eqref{eq:halfopengood}.

\begin{con}\label{con:goodtriang}
Fix a rational polytope $P$ and a rational triangulation $T$ of $\partial P$ into
$(d-1)$-dimensional simplices. We will construct, from this given $T$, a \emph{disjoint}
boundary triangulation into half-open simplices.
Choose any point $\x\in P^\circ$ and let $T'$ be the triangulation of $P$ given by coning over $T$, i.e.,
\[
T' := \left\{\Pyr(\x,\Delta):\Delta\in T \right\}
\]
where for
\[
  \Delta = \left\{ \lambda_1 \v1 + \dots + \lambda_d \vvd : 
  \begin{array}{l}
  \lambda_1, \dots, \lambda_d \ge 0 \\
  \lambda_1 + \dots + \lambda_d = 1
  \end{array}
  \right\} 
\]
we define
\[
  \Pyr(\x,\Delta) := \left\{ \lambda_1 \v1 + \dots + \lambda_d \vvd + \lambda_{d+1} \x : 
  \begin{array}{l}
  \lambda_1, \dots, \lambda_{d+1} \ge 0 \\
  \lambda_1 + \dots + \lambda_{d+1} = 1
  \end{array}
  \right\} .
\]
Choose a point $\y\in P$ that is \textit{generic} with respect to $T'$, i.e., $\y$ is not contained in
any facet-defining hyperplane of any simplex in $T'$. For each simplex $\Delta'\in T'$, remove all
facets of $\Delta'$ that are \emph{visible} from $\y$, i.e., remove all facets $F$ of
$\Delta'$ such that $\y$ is not in the halfspace corresponding to $F$ that defines
$\Delta'$ (see, e.g.,~\cite[Chapter~5]{crt} for more details about half-open triangulations using a
visibility construction). This makes $T'$ into a disjoint triangulation of $P$, which restricts to a disjoint
triangulation of $\partial P$. Moreover, the only closed simplex in the disjoint triangulation of
$\partial P$ is the one corresponding to $\Delta'$ that contains $\y$. 
An example is shown in Figure~\ref{fig:visibilitytriang}.
\end{con}

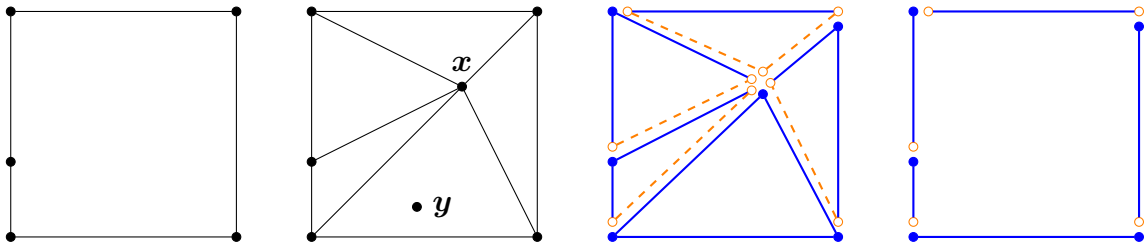
\begin{figure}[htb]
\[
\begin{tikzpicture}
		\node[fill=black, draw=black, shape=circle, scale=0.3] (a) at (0,3) {} ;
		\node[fill=black, draw=black, shape=circle, scale=0.3](ab) at (0,1) {} ;
		\node[fill=black, draw=black, shape=circle, scale=0.3](b) at (0,0) {} ;
		\node[fill=black, draw=black, shape=circle, scale=0.3](c) at (3,3) {} ;
		\node[fill=black, draw=black, shape=circle, scale=0.3](d) at (3,0) {} ;
		\draw (a) -- (b) ;
		\draw (c) -- (d) ;
		\draw (a) -- (c) ;
		\draw (b) -- (d) ;

		\node[fill=black, draw=black, shape=circle, scale=0.3] (a) at (4,3) {} ;
		\node[fill=black, draw=black, shape=circle, scale=0.3](ab) at (4,1) {} ;
		\node[fill=black, draw=black, shape=circle, scale=0.3](b) at (4,0) {} ;
		\node[fill=black, draw=black, shape=circle, scale=0.3](c) at (7,3) {} ;
		\node[fill=black, draw=black, shape=circle, scale=0.3](d) at (7,0) {} ;
		\node[fill=black, draw=black, shape=circle, scale=0.3, label={above:$\x$}](x) at (6,2) {} ;
		\node[fill=black, draw=black, shape=circle, scale=0.3, label={right:$\y$}](q) at (5.4, 0.4) {} ;
		\draw (a) -- (b) ;
		\draw (c) -- (d) ;
		\draw (a) -- (c) ;
		\draw (b) -- (d) ;
		\draw (a) -- (x) ;
		\draw (b) -- (x) ;
		\draw (c) -- (x) ;
		\draw (d) -- (x) ;
		\draw (ab) -- (x) ;

		\node[fill=blue, draw=blue, shape=circle, scale=0.3] (a-closed) at (8,3) {} ;
		\node[fill=none, draw=orange, shape=circle, scale=0.3] (a-open) at (8.2,3) {} ;
		\node[fill=none, draw=orange, shape=circle, scale=0.3](ab-open) at (8,1.2) {} ;
		\node[fill=blue, draw=blue, shape=circle, scale=0.3](ab-closed) at (8,1) {} ;
		\node[fill=blue, draw=blue, shape=circle, scale=0.3](b-closed) at (8,0) {} ;
		\node[fill=none, draw=orange, shape=circle, scale=0.3](b-open) at (8,0.2) {} ;
		\node[fill=none, draw=orange, shape=circle, scale=0.3](c-open) at (11,3) {} ;
		\node[fill=blue, draw=blue, shape=circle, scale=0.3](c-closed) at (11,2.8) {} ;
		\node[fill=none, draw=orange, shape=circle, scale=0.3](d-open) at (11,0.2) {} ;
		\node[fill=blue, draw=blue, shape=circle, scale=0.3](d-closed) at (11,0) {} ;
		\node[fill=blue, draw=blue, shape=circle, scale=0.3](x-b-d) at (10,1.9) {} ;
		\node[fill=none, draw=orange, shape=circle, scale=0.3](x-b-ab) at (9.85,1.95) {} ;
		\node[fill=none, draw=orange, shape=circle, scale=0.3](x-ab-a) at (9.85,2.1) {} ;
		\node[fill=none, draw=orange, shape=circle, scale=0.3](x-a-c) at (10,2.2) {} ;
		\node[fill=none, draw=orange, shape=circle, scale=0.3](x-c-d) at (10.1,2.05) {} ;
		\draw[thick,blue] (b-open) -- (ab-closed) ;
		\draw[thick,blue] (ab-open) -- (a-closed) ;
		\draw[thick, blue] (a-open) -- (c-open) ;
		\draw[thick, blue] (c-closed) -- (d-open) ;
		\draw[thick, blue] (d-closed) -- (b-closed) ;
		\draw[thick, blue] (b-closed) -- (x-b-d) ;
		\draw[thick, blue] (d-closed) -- (x-b-d) ;
		\draw[thick, orange, dashed] (d-open) -- (x-c-d) ;
		\draw[thick, blue] (c-closed) -- (x-c-d) ;
		\draw[thick, orange, dashed] (c-open) -- (x-a-c) ;
		\draw[thick, orange, dashed] (a-open) -- (x-a-c) ;
		\draw[thick, blue] (a-closed) -- (x-ab-a) ;
		\draw[thick, orange, dashed] (ab-open) -- (x-ab-a) ;
		\draw[thick, blue] (ab-closed) -- (x-b-ab) ;
		\draw[thick, orange, dashed] (b-open) -- (x-b-ab) ;

		\node[fill=blue, draw=blue, shape=circle, scale=0.3] (a-closed) at (12,3) {} ;
		\node[fill=none, draw=orange, shape=circle, scale=0.3] (a-open) at (12.2,3) {} ;
		\node[fill=none, draw=orange, shape=circle, scale=0.3](ab-open) at (12,1.2) {} ;
		\node[fill=blue, draw=blue, shape=circle, scale=0.3](ab-closed) at (12,1) {} ;
		\node[fill=blue, draw=blue, shape=circle, scale=0.3](b-closed) at (12,0) {} ;
		\node[fill=none, draw=orange, shape=circle, scale=0.3](b-open) at (12,0.2) {} ;
		\node[fill=none, draw=orange, shape=circle, scale=0.3](c-open) at (15,3) {} ;
		\node[fill=blue, draw=blue, shape=circle, scale=0.3](c-closed) at (15,2.8) {} ;
		\node[fill=none, draw=orange, shape=circle, scale=0.3](d-open) at (15,0.2) {} ;
		\node[fill=blue, draw=blue, shape=circle, scale=0.3](d-closed) at (15,0) {} ;
		\draw[thick,blue] (b-open) -- (ab-closed) ;
		\draw[thick,blue] (ab-open) -- (a-closed) ;
		\draw[thick, blue] (a-open) -- (c-open) ;
		\draw[thick, blue] (c-closed) -- (d-open) ;
		\draw[thick, blue] (d-closed) -- (b-closed) ;
\end{tikzpicture}
\]
\caption{Constructing a half-open boundary triangulation.}\label{fig:visibilitytriang}
\end{figure}

\begin{rmk}\label{rmk:hstarlowerboundedbyopenfacets}
We note that if $T$ is a disjoint lattice triangulation with exactly one closed simplex, then
$\hstar_{\partial P,k}$ is at least the number of simplices in $T$ with $k$ missing faces (since the
fundamental parallelepiped of such a simplex has a lattice point equal to the sum of the $k$
generators opposite those missing faces). In the example in Figure~\ref{fig:visibilitytriang}, if $T$ is a lattice triangulation, we learn that $\hstar_{\partial P}(z)\geq 1+3z+z^2$, coefficient-wise. 
\end{rmk}


\section{A Proof of Theorem~\ref{thm:stapledonrat} via Half-open Pyramids}\label{sec:mainproof}

We will now prove Theorem~\ref{thm:stapledonrat} using a half-open triangulation of the
boundary of the polytope. The geometric interpretation of the coefficients of the
$h^\ast$-polynomial of a half-open simplex in Construction~\ref{con:tiling} allows us to
compare $h^\ast$-polynomials of simplices via their fundamental parallelepipeds, which
will ultimately allow us to compare $h^\ast_P(z)$ and $h^\ast_{\partial P}(z)$. In the
lattice case, we will see the finite geometric series $1+z+\dots+z^{\ell-1}$ arise
naturally as the correction between the denominator of the Ehrhart series of a lattice
polytope and the denominator of the Ehrhart series of a rational polytope with all
integral vertices except for one vertex with denominator $\ell$. In the rational case, we will see the term $\frac{1+z+\dots+z^{\ell-1}}{1+z+\dots+z^{q-1}}$ arise as the correction between the denominator of the Ehrhart series of a rational polytope with denominator $q$ and the denominator of the Ehrhart series of a rational polytope with all denominator-$q$ vertices except for one vertex with denominator~$\ell$.

\begin{proof}[Proof of Theorem~\ref{thm:stapledonrat}]
Let $T$ be a disjoint rational triangulation of $\partial P$ into $(d-1)$-dimensional rational half-open
simplices with denominator $q$. By Construction~\ref{con:goodtriang}, we may assume that exactly one simplex in $T$ is
closed. Let $\x\in P^\circ$ be such that $\ell \x\in (\ell P)^\circ\cap \Z^d$; by the minimality of
$\ell$, we know that $\x$ has denominator $\ell$. 
Given $\Delta \in T$, let $\Pyr(\x,\Delta)$ be the corresponding half-open $d$-simplex.
These simplices give rise to the disjoint rational triangulation $T' := \{\Pyr(\x,\Delta):\Delta\in T\}$ of~$P$. 

We first claim that for any $\Delta \in T$,
\begin{equation}\label{eq:claim}
  \hstar_{\Pyr(\x,\Delta)}(z) = \hstar_{\Delta}(z)+z^\ell \, b_{\Delta}(z)
\end{equation}
for some polynomial $b_{\Delta}(z)$ with nonnegative coefficients.
To see this, let $\v1,\dots,\vvd\in\mathbb{Q}^d$ be the vertices of $\Delta$ (so
$q\v1,\dots,q\vvd\in\Z^d$). Relabel the vertices such that the facets of $\Delta$ opposite $\v1,\dots,\vvr$ are missing and the facets of $\Delta$ opposite $\vr1,\dots,\vvd$ are present in $\Delta$. 

Now we apply Construction~\ref{con:tiling} with the half-open parallelepiped 
\[
  \F_\Delta := \sum_{ j=1 }^{ r } (0,1] \begin{pmatrix}q\vj\\q\end{pmatrix} + \sum_{ j=r+1 }^{
d } [0,1) \begin{pmatrix}q\vj\\q\end{pmatrix}
\]
yielding
\[
\Ehr_\Delta(z)=\frac{\sum_{\m\in \F_\Delta\cap\Z^{d+1}}z^{m_{d+1}}}{(1-z^q)^d} \, .
\]
Meanwhile, $\Pyr(\x,\Delta)$ is a $d$-dimensional simplex with vertices $\v1,\dots,\vvd,\x$, with
missing facets opposite $\v1,\dots,\vvr$ and included facets opposite $\vr1,\dots,\vvd,\x$. Then we again apply Construction~\ref{con:tiling} (noting the choice of fundamental parallelepiped) to obtain a tiling of
\[
  \homg(\Pyr(\x,\Delta)) = \sum_{ j=1 }^{ r } \R_{ >0 } \begin{pmatrix}\vj\\1\end{pmatrix} +
\sum_{ j=r+1 }^{ d } \R_{ \ge 0 } \begin{pmatrix}\vj\\1\end{pmatrix} + \R_{ \ge 0 } \begin{pmatrix}\x\\1\end{pmatrix}
\]
with translates of the half-open parallelepiped 
\[
  \F_{\Pyr(\x,\Delta)} := \sum_{ j=1 }^{ r } (0,1] \begin{pmatrix}q\vj\\q\end{pmatrix} + \sum_{ j=r+1 }^{ d } [0,1) \begin{pmatrix}q\vj\\q\end{pmatrix} + [0,1) \begin{pmatrix} \ell \x \\ \ell \end{pmatrix}
\]
giving
\[
\Ehr_{\Pyr(\x,\Delta)}(z)=\sum_{k_1,\dots,k_{d+1}\in\Z_{\geq0}}\sum_{\m\in\F_{\Pyr(\x,\Delta)}\cap\Z^{d+1} }z^{qk_1+\dots+qk_d+\ell k_{d+1}+m_{d+1}}=\frac{\sum_{\m\in \F_{\Pyr(\x,\Delta)}\cap\Z^{d+1}}z^{m_{d+1}}}{(1-z^q)^d(1-z^\ell)} \, .
\]

Now let $\m$ be a lattice point in $\F_{\Pyr(\x,\Delta)}$, say
\begin{equation}\label{eq:mpoint}
\m=\alpha_1\begin{pmatrix}q\v1\\q\end{pmatrix}+\dots+\alpha_d\begin{pmatrix}q\vvd\\q\end{pmatrix}+\beta\begin{pmatrix}\ell \x\\ \ell\end{pmatrix}
\end{equation}
with $0<\alpha_1,\dots,\alpha_r\leq1$ and $0\leq \alpha_{r+1},\dots,\alpha_d,\beta<1$.
If $\beta=0$, then $\m\in \F_\Delta$.
If $\beta>0$, then the first $d$ coordinates of $\m$ are given by
\[
\alpha_1q\v1+\dots+\alpha_dq\vvd+\beta \ell \x \, ,
\]
and since $\beta>0$, this is a lattice point in $m_{d+1}(\Pyr(\x,\Delta)\setminus \Delta)$, where 
\[
m_{d+1}=\alpha_1q+\dots+\alpha_dq+\beta \ell \, .
\]
But by the minimality of $\ell$, we know that $j(\Pyr(\x,\Delta)\setminus \Delta)$ contains no lattice points for $j=1,\dots,\ell-1$; thus, $m_{d+1}\geq \ell$. 
Therefore, 
\begin{align*}
    \hstar_{\Pyr(\x,\Delta)}(z)&= \sum_{\m\in \F_{\Pyr(\x,\Delta)}\cap\Z^{d+1}}z^{m_{d+1}}=\sum_{\m\in \F_{\Delta}\cap\Z^{d+1}}z^{m_{d+1}} + z^\ell \, b_\Delta(z)\\
    &= \hstar_\Delta(z)+z^\ell \, b_\Delta(z)
\end{align*}
for some polynomial $b_\Delta(z)$ with nonnegative integral coefficients,
and this proves our claim~\eqref{eq:claim}.

Summing over all $\Delta\in T$ now yields
\begin{align*}
    \Ehr_P(z)&=\sum_{\Delta\in T} \Ehr_{\Pyr(\x,\Delta)}(z)=\frac{\sum_{\Delta\in
T} \hstar_{\Pyr(\x,\Delta)}(z)}{(1-z^q)^d(1-z^\ell)}=\frac{\sum_{\Delta\in T}
\left(\hstar_{\Delta}(z)+z^\ell \, b_{\Delta}(z)\right)}{(1-z^q)^d(1-z^\ell)} \, .
\end{align*}
Define $b_P(z):=\sum_{\Delta\in T} b_{\Delta}(z)$, which we know to be a polynomial with nonnegative
integral coefficients. Thus
\begin{align*}
    \Ehr_P(z)&= \frac{\hstar_P(z)}{(1-z^q)^{d+1}} = \frac{\hstar_{\partial P}(z)+z^\ell \, b_P(z)}{(1-z^q)^d(1-z^\ell)}
\end{align*}
and so
\begin{align*}
\frac{1+z+\dots+z^{\ell-1}}{1+z+\dots+z^{q-1}}\, \hstar_P(z)&=\hstar_{\partial P}(z)+z^\ell \, b_P(z)
\, .
\end{align*}
By~\eqref{eq:hstarpalindromic}, $\hstar_{\partial P}(z)$ is palindromic with $z^{qd} \, \hstar_{\partial P}(\frac 1 z)=\hstar_{\partial P}(z)$, and the palindromicity of 
\[
b_P(z)=\frac{\frac{1-z^\ell}{1-z^q} \, \hstar_P(z)-\hstar_{\partial P}(z)}{z^\ell}
\]
follows from this, 
$h_P^\ast(z)=h_{P^\circ}^\ast(z)+(1-z^q)h_{\partial P}^\ast(z)$,
and Ehrhart--Macdonald reciprocity:
\begin{align*}
    z^{qd-\ell}b_P(\tfrac 1 z)&= z^{qd-\ell}\ \frac{\frac{1-\frac 1 {z^\ell}}{1- \frac 1
{z^q}} \, h_P^\ast(\frac 1 z)-\hstar_{\partial P}(\frac 1 z)}{\frac 1 {z^\ell}} \\
    &= z^{qd}\left(\frac{1- \frac 1 {z^\ell}}{1- \frac 1 {z^q}} \Big(h_{P^\circ}^\ast(\tfrac 1 z)+(1- \tfrac 1 {z^q})h_{\partial P}^\ast(\tfrac 1 z)\Big)-h_{\partial P}^\ast(\tfrac 1 z)\right)\\
    &= z^{qd}\left(\frac{1- \frac 1 {z^\ell}}{1- \frac 1 {z^q}}\, h_{P^\circ}^\ast(\tfrac 1 z) +
(1- \tfrac 1 {z^\ell})h_{\partial P}^\ast(\tfrac 1 z)  - h_{\partial P}^\ast(\tfrac 1 z)  \right)\\
    &=\frac{1}{z^\ell}   \left( \frac{z^\ell-1}{z^q-1} \, z^{q(d+1)} \, h_{P^\circ}^\ast(\tfrac 1 z) -
z^{qd} \, h_{\partial P}^\ast(\tfrac 1 z)  \right)\\
    &=\frac{1}{z^\ell}   \left( \frac{1-z^\ell}{1-z^q} \, h_{P}^\ast(z) - h_{\partial P}^\ast(z)  \right)\\
    &= b_P(z) \, . \qedhere
\end{align*}
\end{proof}

\begin{rmk}
If $P \subset \R^{ d-1 }$ is a lattice polytope, it is well known that $\hstar_P(z)$ equals the
$\hstar$-polynomial of $\Pyr(\ed,P) \subset \R^d$. Our above proof implies for more
general pyramids
\[
  \hstar_P(z) \le \hstar_{ \Pyr(\x,P) } (z)
\]
coefficient-wise, for any $\x \in \Z^d$ not in the affine span of~$P$.  Moreover, the proof reveals a rational analogue:
\begin{cor}\label{cor:rationalpyramids}
Let $P\subset \R^{d-1}$ be a rational polytope with denominator $q$. Then 
\[
\hstar_{P}(z)=\hstar_{{\rm Pyr}(\ed,P)}(z) \, ,
\]
where $\hstar_{P}(z)=(1-z^q)^d\Ehr_P(z)$ and
$\hstar_{{\rm Pyr}(\ed,P)}(z)=(1-z^q)^d(1-z)\Ehr_{\Pyr(\ed,P)}(z)$. If $\x\in\frac{1}{r}\Z^d$, then
\[
\hstar_{P}(z)\leq \hstar_{\Pyr(\x,P)}(z) \, ,
\]
where $\hstar_{P}(z)=(1-z^q)^d\Ehr_P(z)$ and
$\hstar_{\Pyr(\x,P)}(z)=(1-z^q)^d(1-z^r)\Ehr_{\Pyr(\x,P)}(z)$.
\end{cor}
\end{rmk}

We finish this section with proving Corollary~\ref{cor:twoleadingcoeff}.

\begin{proof}[Proof of Corollary~\ref{cor:twoleadingcoeff}]
It is well known that $\frac{ \hstar_P(1) }{ d! }$ equals the volume of $P$, which in turn equals $k_d$;
this follows, e.g., by writing $\ehr_P(n)$ in terms of the binomial-coefficient basis
$\binom n d$, $\binom {n+1} d$, \dots, $\binom {n+d} d$ (and then the
coefficients are precisely the coefficients of $\hstar_P(z)$; see, e.g.,
\cite[Section~3.5]{ccd}).
By the same reasoning, $\frac{ \hstar_{\partial P}(1) }{ (d-1)! }$ equals the sum of the volumes of the
facets of $P$, each measures with respect to the sublattice in the affine span of the facet. This sum, in
turn, is well known to equal $2 \, k_{ d-1 }$. Putting it all together, 
\[
  \ell \, d! \, k_d
  = \ell \, \hstar_P(1)
  \ge \hstar_{\partial P}(1)
  = 2 (d-1)! \, k_{ d-1 } \, ,\
\]
where the inequality follows from specializing Theorem~\ref{thm:stapledonrat} (with the interpretation $a(z) = \hstar_{\partial P}(z)$) at $z=1$.
\end{proof}


\section{Rational Reflexive and Gorenstein Polytopes}\label{sec:gorenstein}

A $d$-dimensional polytope $P$ is \textit{reflexive} if it is a lattice polytope that contains the origin in its interior and one of the following (equivalent) statements holds:
\begin{enumerate}
\item the dual polytope of $P$, i.e.,
\[
P^\ast:=\{\x\in\R^d:\x\cdot\y\leq1\text{ for all }\y\in P\} \, ,
\]
 is a lattice polytope;
\item $P=\{\x\in\R^d:\A \x\leq\1\}$, where $\A$ is an integral matrix and $\1$ is a vector of all~1s.
\end{enumerate}

Since $\hstar$-polynomials are invariant under integral translation, we allow translates of reflexive polytopes. Now we need not require that $\0\in P^\circ$, and the following statements are all equivalent to being an integral translate of a reflexive polytope:
\begin{enumerate}
\item $|P^\circ \cap \Z^d| = 1$ and $(t+1)P^\circ\cap\Z^d=tP\cap\Z^d$ for all $t\in\Z_{>0}$;
\item $P$ contains a unique interior lattice point that is lattice distance 1 away from each facet of $P$;
\item $z^d \, \hstar_P(\frac 1 z) = \hstar_P(z) \, .$
\end{enumerate}

If $P$ is a lattice polytope and there exists an integer $g\geq1$ such that $gP$ is reflexive, we say that $P$ is $g$-\textit{Gorenstein}.

By retracing the steps in the proof of Theorem~\ref{thm:stapledonrat} but observing that
there is no $\beta>0$ case if $P$ is reflexive, we can obtain a relationship between
$\hstar_P(z)$ and $\hstar_{\partial P}(z)$. Similarly, we can do this for Gorenstein
polytopes. Both relationships quickly imply that $\hstar_P(z)$ is palindromic since
$\hstar_{\partial P}(z)$ is. In this section we will see these results as
corollaries of more general theorems for \emph{rational} polytopes.

Fiset and Kasprzyk generalized the notion of being reflexive to rational polytopes in \cite{fisetkasprzyk}. They define a polytope $P\subseteq \R^d$ to be \textit{rational reflexive} if it is the convex hull of
finitely many rational points in $\Q^d$, contains the origin in its interior, and has a lattice dual polytope. Equivalently, a rational polytope $P$ is \textit{rational reflexive} if it has a hyperplane description $P=\{\x\in\R^d:\A \x\leq\1\}$, where $\A$ is an integral matrix. 

Fiset and Kasprzyk prove the following:
\begin{thm}
If $P$ is a rational reflexive polytope, then $\hstar_P(z)$ is palindromic.
\end{thm}
In their proof, they define the ``$\hstar$-polynomials'' corresponding to each of the
components of the Ehrhart quasipolynomial of the rational polytope, show (using
reciprocity) that the $i$th such polynomial is equal to the $(q-1-i)$th polynomial with
coefficients read backwards, and finally show that the $\hstar$-polynomial of the rational
polytope can be written in terms of these smaller polynomials. This is a generalized
version of Hibi's proof \cite{hibidual} that lattice polytopes with lattice duals have
symmetric $\hstar$-polynomials. Our proof in the previous section shows the lattice version
in an alternative way, and in this section, we will show the rational version. By
\eqref{eq:hstarpalindromic}, Fiset--Kasprzyk's theorem is an immediate corollary of the following:

\begin{thm}\label{thm:ratrefl}
If $P$ is a rational reflexive polytope with denominator $q$, then 
\[
\hstar_P(z)=\left(1+z+\dots+z^{q-1}\right)\hstar_{\partial P}(z) \, .
\]
\end{thm}
\begin{proof}
Fix a disjoint half-open triangulation $T$ of $\partial P$ into $(d-1)$-dimensional simplices
using the vertices of $P$. Fix $\Delta\in T$, say $\Delta=\conv(\v1,\dots,\vvd)$ with
$\v1,\dots,\vvd\in\frac{1}{q}\Z^d$, where the facets of $\Delta$ opposite $\v1,\dots,\vvr$
are missing and the facets opposite $\vr1,\dots,\vvd$ are included. 
We now follow the steps in our proof of Theorem~\ref{thm:stapledonrat} using
the fundamental parallelepiped
\[
\F_\Delta:=\sum_{j=1}^r (0,1]\begin{pmatrix}
q\vj\\q
\end{pmatrix}+\sum_{j=r+1}^d [0,1)\begin{pmatrix}
q\vj\\q
\end{pmatrix}
\]
yielding
\[
\Ehr_{\Delta}(z)=\frac{\sum_{\m\in\F_\Delta\cap\Z^{d+1}}z^{m_{d+1}}}{(1-z^q)^d}
\]
and
\[
\F_{\Pyr(\0,\Delta)}:=\sum_{j=1}^r (0,1]\begin{pmatrix}
q\vj\\q
\end{pmatrix}+\sum_{j=r+1}^d [0,1)\begin{pmatrix}
q\vj\\q
\end{pmatrix}+[0,1)\begin{pmatrix}
\0\\1
\end{pmatrix}
\]
giving
\[
\Ehr_{\Pyr(\0,\Delta)}(z)=\frac{\sum_{\m\in\F_{\Pyr(\0,\Delta)}\cap\Z^{d+1}}z^{m_{d+1}}}{(1-z^q)^d(1-z)} \, .
\]
Fix a point $\p\in\F_{\Pyr(\0,\Delta)}$, say
\[
\p=\sum_{j=1}^d \alpha_j\begin{pmatrix}
q\vj\\q
\end{pmatrix}+\beta\begin{pmatrix}
\0\\1
\end{pmatrix},
\]
with $0<\alpha_1,\dots,\alpha_r\leq 1$ and $0\leq \alpha_{r+1},\dots,\alpha_d,\beta<1$. The simplex $\Delta$ is contained in some facet-defining hyperplane of $P$, say 
\[
H= \left\{\x\in \R^d: a_1x_1+\dots+a_dx_d=1 \right\},
\]
where $a_1,\dots,a_d\in\Z$. Consider the value of $a_1p_1+\dots+a_dp_d$ (where $p_1,\hdots,p_{d+1}$ are the coordinates of $\p$). If $v_{i,j}$ denotes the $j$th coordinate of $\vi$, then
\begin{align*}
a_1p_1+\dots+a_dp_d&=a_1(\alpha_1qv_{1,1}+\dots+\alpha_dqv_{d,1})+\dots+a_d(\alpha_1qv_{d,1}+\dots+\alpha_dqv_{d,d})\\
&= q\alpha_1(a_1v_{1,1}+\dots+a_dv_{1,d})+\dots+q\alpha_d(a_1v_{d,1}+\dots+v_{d,d}) \, .
\end{align*}
Since each of $\v1,\dots,\vvd\in H$, this is equal to
\[
q\alpha_1+\dots+q\alpha_d=p_{d+1}-\beta \, .
\]
If $\p$ is a \textit{lattice} point in
$\F_{\Pyr(\0,\Delta)}$, then both $a_1p_1+\dots+a_dp_d$ and
$p_{d+1}$ are integers, so $\beta$ must also be an integer. This forces $\beta=0$ and thus
\begin{equation}\label{eq:fundparequal}
\F_{\Pyr(\0,\Delta)}\cap\Z^{d+1}=\F_{\Delta}\cap\Z^{d+1}.
\end{equation}
Therefore,
\[
\Ehr_{\Pyr(\0,\Delta)}(z)=\frac{\hstar_{\Delta}(z)}{(1-z^q)^d(1-z)} \, .
\]
Summing over all $\Delta$, we obtain
\[
\Ehr_P(z)=\frac{\hstar_{\partial P}(z)}{(1-z^q)^d(1-z)}
\]
and so $\hstar_{P}(z)=(1+z+\dots+z^{q-1})\hstar_{\partial P}(z)$.
\end{proof}

Observe that a (lattice) reflexive polytope is a rational reflexive polytope that happens
to be a lattice polytope, so letting $q=1$ yields the following result:
\begin{cor}\label{cor:latrelf}
If $P$ is a (lattice) reflexive polytope, then $\hstar_P(z)=\hstar_{\partial P}(z)$.
\end{cor}

\begin{rmk}
We note that for rational reflexive polytopes (as opposed to lattice reflexive polytopes)
it is not necessarily the case that
$P=\{\x\in\R^d:\A\x\leq\1\}$ implies that the origin is
lattice distance 1 away from each facet-defining hyperplane of $P$. It is possible for a given
facet-defining hyperplane to not contain any lattice points and thus be lattice distance
less than 1 from the origin. However, the equality \eqref{eq:fundparequal} between the sets of lattice points in the two fundamental parallelepipeds still holds.
\end{rmk}

We extend the definition of a rational reflexive to a rational Gorenstein polytope in a natural way. Let $P$ be a rational polytope and let $g\geq1$ be an integer. We say that $P$ is \textit{rational $g$-Gorenstein} if $gP$ is an integral translate of a (lattice) reflexive polytope. We allow $gP$ to be a translate of a reflexive polytope so that we do not force $P$ to have an interior point. Also note that, if we have a rational $g$-Gorenstein polytope with denominator $q$, we must have~$q|g$.

\begin{thm}\label{thm:rationalgorenteinboundary}
If $P$ is a rational $g$-Gorenstein polytope with denominator $q$, then
\[
\hstar_P(z)=\frac{1+z+\dots+z^{q-1}}{1+z+\dots+z^{g-1}} \, \hstar_{\partial P}(z) \, .
\]
\end{thm}

\begin{proof}[Sketch of Proof] The unique interior lattice point of $gP$ is lattice distance $1$ away from the facet-defining hyperplanes of $gP$, so there is some polynomial $f(z)$ such that
\[
\Ehr_{\partial P}(z)=\frac{f(z)}{(1-z^g)^{d}}
\qquad \text{ and }\qquad
\Ehr_P(z)=\frac{f(z)}{(1-z^g)^{d+1}} \, .
\]
On the other hand, $\hstar_{\partial P}(z)$ and $\hstar_{P}(z)$ are such that
\[
\Ehr_{\partial P}(z)=\frac{\hstar_{\partial P}(z)}{(1-z^q)^{d}}
\qquad \text{ and }\qquad
\Ehr_P(z)=\frac{\hstar_{P}(z)}{(1-z^q)^{d+1}} \, .
\]
Therefore,
\begin{align*}
\hstar_{\partial P}(z)
 &=\frac{(1-z^q)^{d}}{(1-z^g)^{d}} \, f(z)
  =\frac{(1-z^q)^{d}}{(1-z^g)^{d}}\, \frac{(1-z^g)^{d+1}}{(1-z^q)^{d+1}} \, \hstar_{P}(z)\\
 &=\frac{1-z^g}{1-z^q}\, \hstar_P(z)
  =\frac{1+z+\dots+z^{g-1}}{1+z+\dots+z^{q-1}} \, \hstar_P(z) \, . \qedhere
\end{align*}
\end{proof}
Observe that a (lattice) $g$-Gorenstein polytope is a rational $g$-Gorenstein polytope that
happens to be lattice, so letting $q=1$ yields the following generalization of
Corollary~\ref{cor:latrelf}.
\begin{cor}
If $P$ is a (lattice) $g$-Gorenstein polytope, then 
\[
\hstar_{\partial P}(z)=(1+z+\hdots+z^{g-1})\hstar_P(z) \, .
\]
\end{cor}
Due to the palindromicity of $\hstar_{\partial P}(z)$, these relations immediately imply that reflexive, Gorenstein, rational reflexive, and rational Gorenstein polytopes all have palindromic $\hstar$-polynomials.


\section{Rational Ehrhart Theory}\label{sec:ratehrhart}

In the last decade, several works have been devoted to an Ehrhart theory of rational polytopes where we
now allow \emph{rational} (or, equivalently, real) dilation factors. The fundamental structure of the Ehrhart
counting function in the rational parameter $\lambda > 0$ is that
\[
  \rationalehr(P;\lambda) := \left| \lambda P \cap \Z^d \right|
\]
is a quasipolynomial in $\lambda$ (i.e., has the same structure as~\eqref{eq:quasidef}
except that now $\lambda$ is a rational variable); this was first shown by Linke \cite{linke}, who proved several other
results about $\rationalehr(P;\lambda)$, including an analogue of Ehrhart--Macdonald
reciprocity; see also
\cite{baldoniberlinekoeppevergnerealdilation,stapledonweightedehrart,stapledonfreesums,royer}.
Our goal in this section is to prove an analogue of Theorem~\ref{thm:stapledonrat} in this setting.

To this end, we first recall \emph{rational Ehrhart series}, which were introduced only
recently~\cite{sophiemattsophia}. Suppose the full-dimensional rational polytope $P \subset \R^d$ is
given by the irredundant halfspace description
\[ 
  P \ = \ \left\{\x\in\R^d : \, \A\,\x\leq \bb \right\},
\]
where $\A\in\Z^{n\times d}$ and $\bb\in\Z^n$ such that the greatest common divisor of
$b_j$ and the entries in the $j$th  row of $\A$ equals $1$, for every $j \in \{1,\dots,n \}$.
We define the \emph{codenominator} $r$ of $P$ to be the least common multiple of the nonzero entries
of $\bb$.
It turns out that $\rationalehr(P;\lambda)$ is fully determined by
evaluations at rational numbers $\lambda$ with denominator $2r$; if $\mathbf \0 \in P$ then we actually need to know only evaluations at rational numbers $\lambda$ with denominator $r$~\cite[Corollary~5]{sophiemattsophia}.
This motivates a study of the two generating functions
\[
\rationalEhr{P}{z} := 1 + \sum_{n \ge 1} \rationalehr(P;\tfrac n r) \, z^{ \frac n r }
\]
and
\[
\refinedrationalEhr{P}{z} := 1 + \sum_{n \ge 1} \rationalehr(P;\tfrac n {2r}) \, z^{ \frac n {2r} } \, ,
\]
which have the following rational form, as shown in~\cite[Theorem~12]{sophiemattsophia}.

\begin{thm}\label{thm:rationalehr}
Let $P$ be a rational $d$-polytope with codenominator~$r$.
\begin{enumerate}[{\rm (a)}]
\item Let $m\in\Z_{>0}$ be such that $\frac{m}{r}P$ is a lattice polytope. Then
\[
\rationalEhr{P}{z}=\frac{\rationalhstar{P}{z}{m}}{(1-z^{\frac m r})^{d+1}}
\]
where $\rationalhstar{P}{z}{m}$ is a polynomial in $\Z[z^{\frac 1 r}]$ of degree $< m(d+1)$ with nonnegative integral coefficients.
\item[{\rm (b)}] Let $m\in\Z_{>0}$ be such that $\frac{m}{2r}P$ is a lattice polytope. Then
\[
\refinedrationalEhr{P}{z}=\frac{\refinedrationalhstar{P}{z}{m}}{(1-z^{\frac m {2r}})^{d+1}}
\]
where $\refinedrationalhstar{P}{z}{m}$ is a polynomial in $\Z[z^{\frac 1 {2r}}]$ of degree $< m(d+1)$ with nonnegative integral coefficients.
\end{enumerate}
\end{thm}

Possibly more important than this theorem are the consequences one can derive from it,
and \cite{sophiemattsophia} (re-)proved several previously-known and novel results in
rational Ehrhart theory. The latter include the facts that $\rationalhstar{P}{z}{m}$ is palindromic if $\0\in
P^\circ$ and that, if $r|m$, extracting the terms with integer exponents from $\rationalhstar{P}{z}{m}$ returns $\hstar_P(z)$, which results in yet another proof of the Betke--McMullen version of Theorem~\ref{thm:stapledonrat} (the $\ell=1$ case).

Theorem~\ref{thm:rationalehr} is based on the observation that
\[
  \rationalhstar{P}{z}{m}=\hstar_{\frac{1}{r}P}(z^{\frac 1 r})
  \qquad \text{ and } \qquad 
  \refinedrationalhstar{P}{z}{m}=\hstar_{\frac{1}{2r}P}(z^{\frac 1 {2r}}); 
\]
note that $m$ is implicitly included in the $\hstar$-polynomial, as we really mean the numerator of the Ehrhart series of $\frac{1}{r}P$ (respectively, $\frac{1}{2r}P$) when the denominator is $(1-z^{m})^{d+1}$. 
The same observation yields the following variant of Theorem~\ref{thm:stapledonrat} for
rational Ehrhart theory.

\begin{cor}
Let $P$ be a rational polytope with codenominator $r$.
\begin{enumerate}[{\rm (1)}]
\item If $\0\in P^\circ$, then $\rationalhstar{P}{z}{m}$ is palindromic.
\item[{\rm (2)}] If $\0\in \partial P$, let $m\in\Z_{>0}$ be such that $\frac{m}{r}P$ is a lattice polytope and let $\ell\geq1$ be the smallest dilate of $\frac{1}{r}P$ that contains an interior point. Then
\[
\frac{1-z^{\frac \ell r}}{1-z^{\frac m r}} \, \rationalhstar{P}{z}{m} = a(z)+z^{\frac
\ell r} \, b(z)
\]
where $a(z) = \rationalhstar{\partial P}{z}{m}$ and $b(z)$ are palindromic polynomials in $\Z[z^{\frac 1 r}]$ with nonnegative integer coefficients.
\item [{\rm (3)}]If $\0\notin P$, let $m\in\Z_{>0}$ be such that $\frac{m}{2r}P$ is a lattice polytope and let $\ell\geq1$ be the smallest dilate of $\frac{1}{2r}P$ that contains an interior point. Then
\[
\frac{1-z^{\frac \ell {2r}}}{1-z^{\frac m {2r}}} \, \refinedrationalhstar{P}{z}{m}=a(z)+z^{\frac \ell {2r}} \, b(z)
\]
where $a(z) = \refinedrationalhstar{\partial P}{z}{m}$ and $b(z)$ are palindromic polynomials in $\Z[z^{\frac 1 {2r}}]$ with nonnegative integer coefficients.
\end{enumerate}
\end{cor}

\begin{exmp}
Let $P=\conv\{(0,0),(0,2),(5,2)\}$, alternatively the intersection of the halfspaces
\[
\{x_1\geq 0\}\cap\{x_2\leq2\}\cap\{5x_2-2x_1\geq0\}.
\]
From this, we can see that $P$ has codenominator $r=2$. Since $\0\in \partial P$, we look at the rational dilate 
\[
\tfrac{1}{2}P=\conv\{(0,0),(0,1),(\tfrac 5 2,1)\}.
\]
The first dilate of $\frac{1}{2}P$ containing an interior lattice point is the $\ell=2$nd dilate, which contains the points $(1,1)$ and $(2,1)$. We choose $m=r=2$ minimal, and we compute the associated $\hstar$-polynomial of $\frac{1}{2}P$:
\[
\hstar_{\frac{1}{2}P}(z)=(1-z^{2})^3 \, \Ehr_{\frac{1}{2}P}(z) =1+4z+7z^2+6z^3+2z^4\]
and so
\begin{align*}
\rationalhstar{P}{z}{2}&=1+4z^{\frac 1 2}+7z+6z^{\frac 3 2}+2z^2 \\
&=(1+4z^{\frac 1 2}+6z+4z^{\frac 3 2}+z^2)+z(1+2z^{\frac 1 2}+z) \, .
\end{align*}
We can check that the first polynomial in the decomposition is equal to
\begin{align*}
\rationalhstar{P}{z}{2}&=\hstar(\partial  \tfrac{1}{2}P;
z^{\frac 1 2})=\frac{h^\ast(\frac{1}{2}P;z^{\frac 1 2})-\hstar(\frac{1}{2}P^\circ;z^{\frac 1
2})}{1-(z^{\frac 1 2})^{2}}\\
&=1+4z^{\frac 1 2}+6z+4z^{\frac 3 2}+z^2.
\end{align*}
Moreover, the power in front of the second polynomial is $z=z^{\frac 2 2}=z^{\frac \ell r}$ and both polynomials are palindromic.\footnote{
We thank Sophie Rehberg for suggesting this example and helping with computing it.
}
\end{exmp}


\bibliographystyle{plain}
\bibliography{bib2}

\end{document}